\documentclass[12pt]{amsart}
\usepackage[utf8]{inputenc}

\let\stdthebibliography\thebibliography
\let\stdendthebibliography\endthebibliography
\renewenvironment*{thebibliography}[1]{%
  \stdthebibliography{LLPSR}}
  {\stdendthebibliography}

\usepackage{microtype,amsmath,amsthm,amssymb,mathscinet}
\usepackage{parskip}
\usepackage{geometry}
\usepackage{xcolor}

\usepackage[hidelinks]{hyperref}
\author{Paul Pollack} 
\address{Department of Mathematics \\ University of Georgia \\ Athens, GA 30602}
\email{pollack@uga.edu}
\thanks{P.P. is supported by NSF award DMS-2001581.}

\author{Lee Troupe} 
\address{Department of Mathematics \\ Mercer University \\ 1501 Mercer University Drive \\ Macon, GA 31207}
\email{troupe\_lt@mercer.edu}


\newcommand\E{\mathbb{E}}
\renewcommand\phi\varphi
\usepackage{accents}

\renewcommand{\pod}[1]{\allowbreak\mathchoice
  {\if@display \mkern 18mu\else \mkern 8mu\fi (#1)}
  {\if@display \mkern 18mu\else \mkern 8mu\fi (#1)}
  {\mkern4mu(#1)}
  {\mkern4mu(#1)}
}
\usepackage{graphicx}
\DeclareMathAlphabet{\curly}{U}{rsfs}{m}{n}

\newtheorem{thm}{Theorem}

\newtheorem{prop}[thm]{Proposition}
\newtheorem{lem}[thm]{Lemma}

\theoremstyle{remark}
\newtheorem*{rmk}{Remark}

\begin{document}
\title[Sums of proper divisors follow Erd\H{o}s--Kac]{Sums of proper divisors follow the Erd\H{o}s--Kac law}
\begin{abstract} Let $s(n)=\sum_{d\mid n,~d<n} d$ denote the sum  of the proper divisors of $n$. The second-named author proved that $\omega(s(n))$ has normal order $\log\log{n}$, the analogue for $s$-values of a classical result of Hardy and Ramanujan. We establish the corresponding Erd\H{o}s--Kac theorem:  $\omega(s(n))$ is asymptotically normally distributed with mean and variance $\log\log{n}$. The same method applies with $s(n)$ replaced by any of several other unconventional arithmetic functions, such as $\beta(n):=\sum_{p\mid n} p$, $n-\phi(n)$, and $n+\tau(n)$ ($\tau$ being the divisor function). 
\end{abstract}
\maketitle
 
\section{Introduction}
Let $s(n)=\sum_{d\mid n,~d<n} d$ denote the sum of the proper divisors of the positive integer $n$, so that $s(n) = \sigma(n)-n$. Interest in the value distribution of $s(n)$ traces back to the ancient Greeks, but the modern study of $s(n)$ could be considered to begin with Davenport \cite{davenport33}, who showed that $s(n)/n$ has a continuous distribution function $D(u)$. Precisely: For each real number $u$, the set of $n$ with $s(n) \le un$ has an asymptotic density $D(u)$ which varies continuously with $u$. Moreover, $D(0)=0$ and $\lim_{u\to\infty} D(u)=1$.

While the values of $\sigma(n) = \prod_{p^e\parallel n} \frac{p^{e+1}-1}{p-1}$ are multiplicatively special, we expect shifting by $-n$ to rub out the peculiarities. That is, we expect the multiplicative statistics of $s(n)$ to resemble those of numbers of comparable size. By Davenport's theorem, it is usually safe to interpret ``of comparable size'' to mean ``of the same order of magnitude as $n$ itself''. 

Various results in the literature validate this expectation about $s(n)$. For example, the first author has shown \cite{pollack14} that $s(n)$ is prime for $O(x/\log{x})$ values of $n\le x$ (and he conjectures that the true count is $\sim x/\log x$, in analogy with the prime number theorem). The second author \cite{troupe20} has proved, in analogy with a classical result of Landau and Ramanujan, that there are $\asymp x/\sqrt{\log{x}}$ values of $n\le x$ for which $s(n)$ is a sum of two squares. Writing $\omega(n)$ for the number of distinct prime factors of $n$, he also showed \cite{troupe15} that $\omega(s(n))$ has normal order $\log\log{n}$. This is in harmony with the classical theorem of Hardy and Ramanujan \cite{HR00} that $\omega(n)$ itself has normal order $\log\log{n}$. 

In this note, we pick back up the study of $\omega(s(n))$. Strengthening the result of \cite{troupe15}, we prove that $\omega(s(n))$ satisfies the conclusion of the Erd\H{o}s--Kac theorem \cite{EK40}.

\begin{thm}\label{thm:main} Fix a real number $u$. As $x\to\infty$, 
\[ \frac{1}{x}\#\{1 < n\le x: \omega(s(n)) - \log\log{x} \le u\sqrt{\log\log{x}}\} \to \frac{1}{\sqrt{2\pi}} \int_{-\infty}^{u} \mathrm{e}^{-\frac{1}{2}t^2}\, \mathrm{d}t.\]
\end{thm}

To prove Theorem \ref{thm:main}, we adapt a simple and elegant proof of the   Erd\H{o}s--Kac theorem due to Billingsley (\cite{billingsley69}, or \cite[pp.\ 395--397]{billingsley95}). Making this go requires estimating, for squarefree $d$, the number of $n\le x$ for which $d\mid s(n)$. A natural attack on this latter problem is to break off the largest prime factor $P$ of $n$, say $n=mP$. Most of the time, $P$ does not divide $m$, so that $\sigma(n)=\sigma(m)(P+1)$. Then asking for $d$ to divide $s(n)$ amounts to imposing the congruence $s(m) P\equiv -\sigma(m)\pmod{d}$. For a given $m$, the corresponding $P$ are precisely those that lie in a certain interval $I_m$ and a certain (possibly empty) set of arithmetic progressions. At this point we adopt (and adapt) a strategy of Banks, Harman, and Shparlinski \cite{BHS05}. Rather than analytically estimate the number of these $P$, we relate the count of such $P$ back to the total number of primes in the interval $I_m$, which we leave unestimated! This allows one to avoid certain losses of precision in comparison with \cite{troupe15}. A similar strategy was used recently in \cite{LPR21}  to show that $s(n)$, for composite  $n\le x$, is asymptotically uniformly distributed modulo primes $p\le (\log{x})^A$ (with $A>0$ arbitrary but fixed).

Our proof of Theorem \ref{thm:main} is fairly robust. In the final section, we describe the modifications necessary to prove the corresponding result with $s(n)$ replaced by $\beta(n):= \sum_{p\mid n} p$, $n-\phi(n)$, or $n+\tau(n)$, where $\tau(n)$ is the usual divisor-counting function.

For other recent work on the value distribution of $s(n)$, see  \cite{LP15,PP16, pomerance18, PPT18}.

\subsection*{Notation} Throughout, the letters $p$ and $P$ are reserved for primes. We write $(a,b)$ for the greatest common divisor of $a, b$. We let $P^{+}(n)$ denote the largest prime factor of $n$, with the convention that $P^{+}(1)=1$. We write $\log_k$ for the $k$th iterate of the natural logarithm. We use $\mathbb{E}$ for expectation and $\mathbb{V}$ for variance. 

\section{Outline} We let $x$ be a large real number and we work on the probability space
\[ \Omega:= \{n \le x: n\text{ composite}, P^{+}(n) > x^{1/\log_4 x}, \text{ and } P^{+}(n)^2\nmid n\}, \]
equipped with    the uniform measure. Standard arguments (compare with the proof of Lemma 2.2 in \cite{troupe15}) show that as $x\to\infty$,
\[ \#\Omega = (1+o(1)) x. \]

We let $y=(\log{x})^2$ and $z=x^{1/\log_3{x}}$, and we define
\[ \mathcal{P} = \{\text{primes $p$ with $y < p \le z$}\}. \]

For each prime $p\le x^2$, we introduce the random variable $X_p$ on $\Omega$ defined by
\[ X_p(n) =
\begin{cases} 1 & \text{if $p\mid s(n)$}, \\
0 &\text{otherwise}.
\end{cases}
\]
We let $Y_p$ be Bernoulli random variables which take the value $1$ with probability $1/p$. We define
\[ X = \sum_{p \in \mathcal{P}} X_p\quad\text{and}\quad Y = \sum_{p \in \mathcal{P}} Y_p; \]
we think of $Y$ as an idealized model of $X$.

Observe that 
\begin{align} \mu:=\mathbb{E}[Y] = \sum_{p \in \mathcal{P}} \frac{1}{p} &= \log\log{z} - \log\log{y} + o(1)\notag \\
&= \log\log{x} + o(\sqrt{\log\log{x}})\label{eq:muest} \end{align}
while
\begin{equation}\label{eq:sigmaest} \sigma^2:= \mathbb{V}[Y] =\sum_{p \in \mathcal{P}} \frac{1}{p}\left(1-\frac{1}{p}\right) \sim \log\log{x}. \end{equation}
We renormalize $Y$ to have mean $0$ and variance $1$ by defining
\[ \tilde{Y} = \frac{Y-\mu}{\sigma}. \]

\begin{lem}\label{lem:probability}  $\tilde{Y}$ converges in distribution to the standard normal $\mathcal{N}$, as $x\to\infty$. Moreover, $\mathbb{E}[\tilde{Y}^k] \to \mathbb{E}[\mathcal{N}^k]$ for each fixed positive integer $k$. 
\end{lem}
\begin{proof}[Proof {\rm(}sketch{\rm)}] Both claims follow from the proof in \cite[pp. 391--392]{billingsley95} of the central limit theorem through the method of moments. One needs only that the recentered variables $Y_p':=Y_p-\frac{1}{p}$, for $p \in \mathcal{P}$, are independent mean 0 variables of finite variance, bounded by $1$ in absolute value, with  $\sum_{p \in \mathcal{P}} \mathbb{V}[Y_p'] \to\infty$ as $x\to\infty$. (Note that $\sum_{p \in \mathcal{P}} \mathbb{V}[Y_p'] = \sum_{p \in \mathcal{P}} \mathbb{V}[Y_p] = \mathbb{V}[Y] = \sigma^2$ in our above notation.)
\end{proof}

Let $X = \sum_{p \in \mathcal{P}} X_p$ and $\tilde{X} = \frac{X-\mu}{\sigma}$. The next section is devoted to the proof of the following proposition. 

\begin{prop}\label{prop:correctmoments} For each fixed positive integer $k$, 
\[ \E[\tilde{X}^k] - \E[\tilde{Y}^k] \to 0. \]
\end{prop}

Lemma \ref{lem:probability} and Proposition \ref{prop:correctmoments} imply that $\E[\tilde{X}^k] \to \E[\mathcal{N}^k]$, for each $k$. So by the  method of moments \cite[Theorem 30.2, p. 390]{billingsley95}, $\tilde{X} = \frac{X-\mu}{\sigma}$ converges in distribution to the standard normal. 

This is most of the way towards  Theorem \ref{thm:main}. Since $\#\Omega = (1+o(1)) x$ and $\mu, \sigma$ satisfy the estimates \eqref{eq:muest}, \eqref{eq:sigmaest}, Theorem \ref{thm:main} will follow if we show that $\frac{\omega(s(\cdot)) - \mu}{\sigma}$ (viewed as a random variable on $\Omega$) converges in distribution to the standard normal. Observe that $s(n) \le \sum_{d < n} d < n^2 \le x^2$ for every $n \le x$. So defining $X^{(s)} = \sum_{p \le y} X_p$ and $X^{(l)} = \sum_{z < p \le x^2} X_p$, we have $\omega(s(.)) = X^{(s)} + X + X^{(l)}$ on $\Omega$ and 
\[ \frac{\omega(s(\cdot)) - \mu}{\sigma} = \tilde{X} + \frac{X^{(s)}}{\sigma} + \frac{X^{(\ell)}}{\sigma}. \]
Since $\tilde{X}$ converges to the standard normal, to complete the proof of Theorem \ref{thm:main} it suffices to show that $\frac{X^{(s)}}{\sigma}$ and $\frac{X^{(\ell)}}{\sigma}$ converge to $0$ in probability (see \cite[Theorem 25.4, p.\ 332]{billingsley95}). Convergence to $0$ in probability is obvious for $X^{(\ell)}/\sigma$: A positive integer not exceeding $x^2$ has at most $\frac{\log{(x^2)}}{\log{z}} = 2\log_3 x$ prime exceeding $z$, so that 
\[ |X^{(\ell)}/\sigma| \le 2\log_3{x}/\sigma = o(1) \]
on the entire space $\Omega$. Since $\sigma \sim \sqrt{\log\log{x}}$, that $X^{(s)}/\sigma$ tends to $0$ in probability follows from the next lemma together with Markov's inequality.

\begin{lem}\label{lem:smallprimes} $\E[X^{(s)}] \ll \log_3{x} \log_4 x$ for all large $x$.
\end{lem}

\begin{proof} Put $L=x^{1/\log_4 x}$, and for each $m\le x$, let $L_m = \max\{x^{1/\log_4 x}, P^{+}(m)\}$. The $n$ belonging to $\Omega$ are precisely the positive integers $n$ that admit a decomposition $n=mP$, where $m > 1$ and $L_m < P \le x/m$. Note that this decomposition of $n$ is unique whenever it exists, since one can recover $P$ from $n$ as $P^{+}(n)$. 

Let $n \in \Omega$ and write $n=mP$ as above. Then $s(mP) = \sigma(m)(P+1)-mP = Ps(m)+\sigma(m)$. Hence, for each $p \le y$,  
\[ \sum_{n \in \Omega} X_p(n) = \sum_{\substack{n \in \Omega \\ p \mid s(n)}} 1 = \sum_{1 < m < x/L} \sum_{\substack{L_m < P \le x/m \\ Ps(m) \equiv - \sigma(m)\pmod{p}}} 1. \]
If $p\nmid s(m)$, then the congruence $Ps(m)\equiv -\sigma(m)\pmod{p}$ puts $P$ in a determined  congruence class mod $p$ (possibly $0\bmod p$). By Brun--Titchmarsh, the number of such $P \le x/m$ is 
\[ \ll \frac{x}{mp \log(x/mp)} \ll \frac{x\log_4 x}{mp\log{x}}.\]
(We use here that $x/mp > L/p > L^{1/2}$ and $\log(L^{1/2}) \gg \log{x}/\log_4 x$.) On the other hand, if $p\mid s(m)$, then the congruence $Ps(m)\equiv -\sigma(m)\pmod{p}$ has integer solutions $P$ only when $p\mid \sigma(m)$, in which case $p\mid \sigma(m)-s(m)=m$. In that scenario, every prime $P$ satisfies $Ps(m)\equiv -\sigma(m)\pmod{p}$. Since there are $\ll \frac{x \log_4 x}{m\log{x}}$  primes $P\le x/m$, we conclude that 
\[ \sum_{1 < m < x/L} \sum_{\substack{L_m < P \le x/m \\ Ps(m) \equiv - \sigma(m)\pmod{p}}} 1 \ll \sum_{\substack{m \le x \\ p\mid m}} \frac{x\log_4{x}}{m\log{x}} + \sum_{m\le x} \frac{x\log_4 x}{mp\log{x}} \ll \frac{x\log_4{x}}{p}. \]
Keeping in mind that $|\Omega| \sim x$,
\[ \E[X^{(s)}] \ll \frac{1}{x} \sum_{p \le y} \frac{x\log_4{x}}{p} \ll \log_4 x \log_2 y \ll \log_4 x\log_3 x. \qedhere\]
\end{proof}

\section{Completion of the proof of Theorem \ref{thm:main}: Proof of Proposition \ref{prop:correctmoments}}
Throughout this section, $k$ is a fixed positive integer. All estimates are to be understood as holding for $x$ large enough, allowed to depend on $k$, and implied constants in Big-oh relations and $\ll$ symbols may depend on $k$.

Recalling the definitions of $\tilde{X}$, $\tilde{Y}$ and expanding,
\begin{align*} \E[\tilde{X}^k] - \E[\tilde{Y}^k] &= \frac{1}{\sigma^k} \sum_{j=1}^{k}\binom{k}{j}(-\mu)^{k-j} (\E[X^j]-\E[Y^j]) \\ &\ll (\log_2{x})^{O(1)} \sum_{j=1}^{k} |\E[X^j]-\E[Y^j]|.
\end{align*}
For each $j=1,2,\dots,k$,
\[ \E[X^j] - \E[Y^j] = \sum_{p_1,\dots,p_j \in \mathcal{P}} \left(\E[X_{p_1} \cdots X_{p_j}] - \E[Y_{p_1} \cdots Y_{p_j}]\right). \]
Writing $d$ for the product of the distinct primes from the list $p_1,\dots,p_j$, we have $X_{p_1} \cdots X_{p_j} = \prod_{p\mid d} X_p$, $Y_{p_1}\cdots Y_{p_j} = \prod_{p\mid d} Y_p$, and 
\[ \E[X_{p_1} \cdots X_{p_j}] - \E[Y_{p_1} \cdots Y_{p_j}] = \frac{1}{|\Omega|} \sum_{\substack{n \in \Omega \\ d\mid s(n)}} 1 - \frac{1}{d}.  \] 
Observe that given $d$ and $j$, there are only $O(1)$ possibilities for the  original list $p_1,\dots,p_j$. Since there are $O(1)$ possibilities for $j$, we conclude that 
\[ \E[\tilde{X}^k] - \E[\tilde{Y}^k] \ll (\log_2{x})^{O(1)} \sum_{\substack{d \text{ squarefree} \\ p \mid d \Rightarrow p \in \mathcal{P} \\ \omega(d) \le k}} \bigg|\frac{1}{|\Omega|}\sum_{\substack{n \in \Omega \\ d\mid s(n)}} 1 - \frac{1}{d}\bigg|.  \]
We will show that 
\begin{equation}\label{eq:distribution} \sum_{\substack{d \text{ squarefree} \\ p \mid d \Rightarrow p \in \mathcal{P} \\ \omega(d) \le k}} \bigg|\frac{1}{|\Omega|}\sum_{\substack{n \in \Omega \\ d\mid s(n)}} 1 - \frac{1}{d}\bigg| \ll \frac{(\log_2 x)^{O(1)}}{\log{x}}. \end{equation}
Hence, $\E[\tilde{X}^k] - \E[\tilde{Y}^k] \to 0$ as claimed.

Let $d$ be a product of at most $k$ distinct primes from $\mathcal{P}$. It will be useful in subsequent arguments to keep in mind that $d = x^{O(1/\log_3 x)}$, and so is of size $L^{o(1)}$. Decomposing each $n \in \Omega$ in the form $mP$, as in the proof of Lemma \ref{lem:smallprimes}, we see that
\begin{equation}\label{eq:sumddiv} \sum_{\substack{n \in \Omega \\ d\mid s(n)}} 1 = \sum_{1 < m < x/L} \sum_{\substack{L_m < P \le x/m \\ Ps(m) + \sigma(m) \equiv 0\pmod{d}}} 1,\end{equation}
where as before $L= x^{1/\log_4 x}$ and $L_m = \max\{x^{1/\log_4 x}, P^{+}(m)\}$. To analyze the right-hand double sum, we consider various cases for $m$.

Say that $m$ is \textsf{$d$-compatible} if for every prime $p$ dividing $d$, either $p$ divides both $s(m)$ and $\sigma(m)$ or $p$ divides neither. Then $m$ is $d$-compatible precisely when the congruence $us(m) + \sigma(m) \equiv 0\pmod{d}$ has a solution $u$ coprime to $d$; in this case, the primes $P$ with $Ps(m)+\sigma(m)\equiv 0\pmod{d}$ are precisely those belonging to a certain coprime residue class modulo $d/(d,s(m))$. We call $m$ \textsf{$d$-ideal} if $\gcd(d,s(m)\sigma(m))=1$; equivalently, $m$ is $d$-ideal if $m$ is $d$-compatible and $\gcd(d,s(m))=1$. Note that only $d$-compatible values of $m$ contribute to the right side of \eqref{eq:sumddiv}. 

When $m$ is $d$-ideal, 
\[ \sum_{\substack{L_m < P \le x/m \\ Ps(m) + \sigma(m) \equiv 0\pmod{d}}} 1 = \frac{1}{\phi(d)}  \sum_{L_m < P \le x/m} 1  + O(E(x/m;d)),\]
where
\[ E(T;q) := \max_{2\le t \le T} \max_{\gcd(a,q)=1} \left|\pi(t;q,a)- \frac{\pi(t)}{\phi(q)} \right|. \]
So the contribution to the right-hand side of \eqref{eq:sumddiv} from $d$-ideal $m$ is 
\begin{multline} \frac{1}{\phi(d)}\sum_{1 < m < x/L}\sum_{L_m < P \le x/m} 1 - \frac{1}{\phi(d)} \sum_{\substack{1 < m < x/L \\\text{not $d$-ideal}}}\sum_{L_m < P \le x/m}1+ O\left(\sum_{m < x/L} E(x/m; d)\right) \\
= \frac{|\Omega|}{\phi(d)} - \frac{1}{\phi(d)} \sum_{\substack{1 < m < x/L \\\text{not $d$-ideal}}}\sum_{L_m < P \le x/m}1+ O\left(\sum_{m < x/L} E(x/m; d)\right).\label{eq:dideal1}\end{multline}
Since $d$ is a product of $O(1)$ primes all of which exceed $y$, the first main term here admits the estimate \begin{equation}\label{eq:dideal2} \frac{|\Omega|}{\phi(d)} = \frac{|\Omega|}{d}\left(1+O(1/y)\right) = \frac{|\Omega|}{d} + O(x/dy).\end{equation}
We bound the second main term, involving the double sum on $m,P$, from above. The inner sum is no more than $\pi(x/m) \ll \frac{x}{m\log(x/m)} \ll \frac{x\log_4 x}{m\log{x}}$, so that 
\begin{equation}\label{eq:dideal3} \frac{1}{\phi(d)}\sum_{\substack{1 < m < x/L \\\text{not $d$-ideal}}}\sum_{L_m < P \le x/m}1 \ll \frac{x\log_4 x}{\log{x}} \sum_{\substack{1 < m < x/L \\ \text{not $d$-ideal} }}\frac{1}{md}. \end{equation}

Next, we investigate the contribution to the right-hand side of \eqref{eq:sumddiv} from $m$ that are $d$-compatible but not $d$-ideal. For these $m$, the corresponding primes $P$ are restricted to a progression mod $d/(d,s(m))$, and so by the Brun--Titchmarsh inequality these $m$ contribute 
\begin{equation}\label{eq:dcompatnotideal} \ll \sum_{\substack{1 < m < x/L \\ \text{ $d$-compat} \\ \text{not $d$-ideal}}} \frac{x}{m \cdot \phi(d/(d,s(m))) \log(x (d,s(m))/md)} \ll \frac{x\log_4{x}}{\log{x}} \sum_{\substack{1 < m < x/L \\ \text{ $d$-compat} \\ \text{not $d$-ideal}}} \frac{(d,s(m))}{md}. \end{equation}

We derive from \eqref{eq:dideal1}, \eqref{eq:dideal2}, \eqref{eq:dideal3}, and \eqref{eq:dcompatnotideal} that 
\begin{equation*}
\bigg|\frac{1}{|\Omega|}\sum_{\substack{n \in \Omega \\ d\mid s(n)}} 1 - \frac{1}{d}\bigg| \ll \frac{1}{x}\sum_{m< x/L}|E(x/m;d )| + \frac{1}{dy} + \frac{\log_4{x}}{\log{x}}\sum_{\substack{1 < m < x/L \\ \text{not $d$-ideal}}}\frac{1}{md} + \frac{\log_4 x}{\log{x}} \sum_{\substack{1 < m < x/L \\ \text{ $d$-compat} \\ \text{not $d$-ideal}}} \frac{(d,s(m))}{md}.
\end{equation*}
Now we sum on $d$.

First off, the Bombieri--Vinogradov theorem implies that
\begin{align*} \sum_{\substack{d\text{ squarefree} \\p \mid d \Rightarrow p \in \mathcal{P} \\ \omega(d)\le k}} \left(\frac{1}{x} \sum_{m < x/L} |E(x/m; d)|\right) &\le \frac{1}{x}\sum_{m < x/L} \sum_{d \le (x/m)^{1/3}} |E(x/m;d)| \\
&\ll \frac{1}{x} \sum_{m < x/L} \frac{x/m}{(\log{(x/m)})^2} \ll \frac{(\log_4 x)^2}{(\log{x})^2} \sum_{m < x/L}\frac{1}{m} \ll \frac{(\log_4 x)^2}{\log{x}}.
\end{align*}
Next,
\[ \sum_{\substack{d\text{ squarefree} \\p \mid d \Rightarrow p \in \mathcal{P} \\ \omega(d)\le k}} \frac{1}{dy} \le \frac{1}{y} \sum_{j=0}^{k} \frac{1}{j!}\left(\sum_{p \in \mathcal{P}} \frac{1}{p}\right)^j \ll \frac{(\log_2 x)^k}{(\log{x})^2}.  \]
Continuing, note that if $m$ is not $d$-ideal, then there is a prime $p\mid d$ with $p\mid s(m) \sigma(m)$. Hence,
\begin{multline*} \sum_{\substack{d\text{ squarefree} \\p \mid d \Rightarrow p \in \mathcal{P} \\ \omega(d)\le k}} \left(\frac{\log_4 x}{\log x} \sum_{\substack{1 < m < x/L \\ \text{not $d$-ideal}}}\frac{1}{md}\right) \le \frac{\log_4 x}{\log x} \sum_{\substack{d\text{ squarefree} \\p \mid d \Rightarrow p \in \mathcal{P} \\ \omega(d)\le k}} \frac{1}{d} \sum_{p\mid d}  \sum_{\substack{1 < m< x/L \\ p \mid s(m)\sigma(m)}} \frac{1}{m} \\
\le \frac{\log_4 x}{\log x} \sum_{1 < m < x/L}\frac{1}{m} \sum_{\substack{p \mid s(m)\sigma(m) \\ p \in \mathcal{P}}} \sum_{\substack{d \le x\text{ squarefree} \\ p \mid d \\ \omega(d) \le k}} \frac{1}{d} \ll \frac{(\log_2 x)^{O(1)}}{\log{x}} \sum_{1 < m < x/L}\frac{1}{m} \sum_{\substack{p \mid s(m)\sigma(m) \\ p \in \mathcal{P}}}\frac{1}{p}. \end{multline*}
Since each $p \in \mathcal{P}$ exceeds $y$, the final sum on $p$ is $\ll \omega(s(m)\sigma(m))/y \ll \log{x}/y = 1/\log{x}$, and so the last displayed expression is
\[ \ll \frac{(\log_2 x)^{O(1)}}{(\log{x})^2} \sum_{1 < m < x/L}\frac{1}{m} \ll \frac{(\log_2 x)^{O(1)}}{\log{x}}. \]
Finally, suppose $m$ is $d$-compatible but not $d$-ideal. Then $(d,s(m)) > 1$, $(d,s(m)) \mid \sigma(m)$, and $(d,s(m)) \mid \sigma(m)-s(m)=m$. Hence, thinking of $d'$ as $(d,s(m))$, 
\begin{equation} \sum_{\substack{d\text{ squarefree} \\p \mid d \Rightarrow p \in \mathcal{P} \\ \omega(d)\le k}} \Bigg(\frac{\log_4 x}{\log{x}} \sum_{\substack{1 < m < x/L \\ \text{ $d$-compat} \\ \text{not $d$-ideal}}} \frac{(d,s(m))}{md}\Bigg) \le \frac{\log_4 x}{\log{x}} \sum_{\substack{d\text{ squarefree} \\p \mid d \Rightarrow p \in \mathcal{P} \\ \omega(d)\le k}}\frac{1}{d} \sum_{\substack{d' \mid d \\ d'>1}} d' \sum_{\substack{1 < m < x/L \\ d' \mid m,~d'\mid \sigma(m)}}\frac{1}{m}.\label{eq:toinsertback}\end{equation}
Let us estimate the inner sum on $m$. Write $m = d' m'$. The contribution to that sum from $m$ with $(d',m') > 1$ is at most 
\[ \sum_{p\mid d'} \frac{1}{d'}\sum_{\substack{m' < x \\ p \mid m'}} \frac{1}{m'} \ll \frac{1}{d'}\log{x} \sum_{p \mid d'} \frac{1}{p} \ll \frac{\log{x}}{d'y} \omega(d') \ll \frac{1}{d'\log{x}}. \]
Suppose now that $\gcd(d',m')=1$. If $d' \mid \sigma(m)$, then $P^{+}(d') \mid \sigma(d') \sigma(m')$, while $P^{+}(d') > P^{+}(\sigma(d'))$ (since $d'$ is a squarefree product of odd primes and $d'>1$). Thus, $P^{+}(d') \mid \sigma(m')$. Choose a prime power $q^e\parallel m'$ with $P^{+}(d')\mid \sigma(q^e)$. If $e\ge 2$, then $y < P^{+}(d') \le \sigma(q^e) < 2q^e$, and so $m'$ has squarefull part exceeding $y/2$. If $e=1$, then $q\parallel m'$ with $q\equiv -1\pmod{P^{+}(d')}$. Hence, these $m$ make a contribution to the inner sum bounded by
\begin{multline*} \frac{1}{d'}\Bigg(\sum_{\substack{r > y/2 \\ \text{squarefull}}}\sum_{\substack{m' < x \\ r \mid m'}}\frac{1}{m'} + \sum_{p \mid d'} \sum_{\substack{q < x \text{ prime} \\ q\equiv -1\pmod{p}}} \sum_{\substack{m' < x \\ q \mid m'}}\frac{1}{m'}\Bigg) \ll \frac{\log{x}}{d'} \Bigg(\sum_{\substack{r > y/2 \\ \text{squarefull}}}\frac{1}{r} + \sum_{p \mid d'} \sum_{\substack{q < x \text{ prime} \\ q\equiv -1\pmod{p}}}\frac{1}{q} \Bigg)  \\
\ll \frac{\log{x}}{d'} \Bigg(\frac{1}{\log{x}}+ \sum_{p \mid d'} \frac{\log{x}}{p} \Bigg) \ll \frac{1}{d'} + \frac{(\log{x})^2}{d'} \sum_{p \mid d'}\frac{1}{p} \ll \frac{1}{d'} + \frac{(\log{x})^2}{d'} \frac{1}{y} \ll \frac{1}{d'}. \end{multline*}
Inserting these estimates back above, the right-hand side of  \eqref{eq:toinsertback} is seen to be 
\[ \ll \frac{\log_4{x}}{\log x} \sum_{\substack{d\text{ squarefree} \\p \mid d \Rightarrow p \in \mathcal{P} \\ \omega(d)\le k}}\frac{1}{d} \sum_{\substack{d' \mid d \\ d'>1}} 1 \ll  \frac{\log_4{x}}{\log x} \sum_{\substack{d\text{ squarefree} \\p \mid d \Rightarrow p \in \mathcal{P} \\ \omega(d)\le k}}\frac{1}{d} \ll \frac{(\log_2{x})^{O(1)}}{\log{x}}.\]
Assembling the last several estimates yields \eqref{eq:distribution}, which completes the proof of Theorem \ref{thm:main}. 

\begin{rmk} As with most variants of Erd\H{o}s--Kac, Theorem \ref{thm:main} remains valid if we count prime factors with multiplicity. Define $\omega'(n) = \sum_{p^k\parallel n} k$. (We avoid the more familiar notation $\Omega(n)$, since $\Omega$ denotes  our sample space.) It is shown in \cite{troupe15} that, for a certain subset $\Omega'$ of $(1,x]$ containing $(1+o(1))x$ elements,
\[ \frac{1}{x} \sum_{n \in \Omega'(x)} (\omega'(s(n)) - \omega(s(n))) \ll (\log_3{x})^2. \]
(See p.\ 133 of \cite{troupe15}.)
It follows that away from a set of $o(x)$ elements of $(1,x]$, we have $\omega'(s(n)) - \omega(s(n)) < (\log\log{x})^{0.49}$ (say). Hence, the Erd\H{o}s--Kac theorem for $\omega'(s(n))$ is a consequence of the corresponding theorem for $\omega(s(n))$.
\end{rmk}

\section{Other arithmetic functions}

The astute reader will observe that many of the calculations above do not depend on properties specific to $s(n)$. In this section, we discuss how to adapt the previous argument for other arithmetic functions.

Let $f$ be an integer-valued arithmetic function with $f(n)$ nonzero for $n>1$ and $|f(n)| \leq x^{O(1)}$ for all $n \leq x$. Assume that for all positive integers $m$ and all primes $P$ not dividing $m$, there are integers $a(m)$ and $b(m)$ such that $f(mP) = Pa(m) + b(m)$, with  $a(m),b(m)$ nonzero for $m>1$. Finally, assume that $|a(m)|, |b(m)| \leq x^{O(1)}$ for all $1 < m \leq x$. (For $f(n) = s(n)$, we have $0 < s(n) \leq x^2$ when $1 < n \leq x$, and  $s(mP) = Ps(m) + \sigma(m)$ for any positive integer $m$ and any prime $P \nmid m$.) All symbols are defined as in Section 2, except that the random variable $X_p$ is now equal to 1 if $p \mid f(n)$ and is 0 otherwise.

To obtain an Erd{\H o}s--Kac-type result for $\omega(f(n))$, we follow the same general strategy as in the case $f(n) = s(n)$. By the method of moments, Lemma \ref{lem:probability} and the analogue of Proposition \ref{prop:correctmoments} (once shown) will establish that $\tilde{X} = \frac{X-\mu}{\sigma}$ converges in distribution to the standard normal. Recall that $y = (\log x)^2$ and $z = x^{1/\log_3 x}$; then
\[
\frac{\omega(f(\cdot)) - \mu}{\sigma} = \tilde{X} + \frac{X^{(s)}}{\sigma} + \frac{X^{(l)}}{\sigma},
\]
where $X^{(s)} = \sum_{p \leq y} X_p$ and $X^{(l)} = \sum_{z < p \leq x^c} X_p$, where $c > 0$ is a constant such that $|f(n)| \leq x^c$ for all $n \leq x$.

As before, our task is to show that $\frac{X^{(s)}}{\sigma}$ and $\frac{X^{(l)}}{\sigma}$ converge to 0 in probability. The argument for $X^{(l)}$ is the same, with the exponent $2$ replaced by $c$. For $X^{(s)}$, we again hope to use Markov's inequality coupled with an upper bound for $\E[X^{(s)}]$ of size $o(\sqrt{\log_2 x})$, analogous to Lemma \ref{lem:smallprimes}. The argument there yields, in this case,
\[
\E[X^{(s)}] \ll \log_3 x \log_4 x + \frac{\log_4 x}{\log x} \sum_{p \leq y} \sum_{\substack{m \leq x \\ p \mid a(m) \text{ and } p \mid b(m)}} \frac{1}{m}. 
\]
Thus, the aim is to show that
\[
\sum_{p \leq y} \sum_{\substack{m \leq x \\ p \mid a(m) \text{ and } p \mid b(m)}} \frac{1}{m} = o\Bigg( \frac{\sqrt{\log_2 x}}{\log_4 x} \log{x}\Bigg).
\]

We now turn our attention to the analogue of Proposition \ref{prop:correctmoments}. Say that $m$ is $d$-compatible if for every $p \mid d$, either $p$ divides both $a(m)$ and $b(m)$ or $p$ divides neither; and $m$ is $d$-ideal if $\gcd(d, a(m)b(m)) = 1$. Equivalently, $m$ is $d$-ideal if $m$ is $d$-compatible and $\gcd(d, a(m)) = 1$. Tracing through the argument in Section 3, we see that few of the calculations depend on specific properties of $f(n)$; in fact, the analogue of Proposition \ref{prop:correctmoments} is established if
\[
\sum_{\substack{d\text{ squarefree} \\p \mid d \Rightarrow p \in \mathcal{P} \\ \omega(d)\le k}} \sum_{\substack{1 < m < x \\ \text{ $d$-compat} \\ \text{not $d$-ideal}}} \frac{(d,a(m))}{md} \ll (\log_2 x)^{O(1)}.
\]

We summarize the above discussion in the following proposition.

\begin{prop}\label{prop:ekgeneral}
Suppose $f(n)$ is an integer-valued arithmetic function with $f(n)$ nonzero when $n>1$ and $|f(n)| \leq x^{O(1)}$ for all $n \leq x$. Suppose also that for every positive integer $m$, there are $a(m)$ and $b(m)$ such that
\[
f(mP) = Pa(m) + b(m) \text{ for all primes $P$ not dividing $m$},
\]
and that 
\[
|a(m)|, |b(m)| \leq x^{O(1)}\quad\text{whenever $m\le x$}.
\]
Suppose also that $a(m), b(m)$ are nonzero whenever $m>1$. Then, if 
\begin{equation}\label{smallprimessum}
\sum_{p \leq y} \sum_{\substack{m \leq x \\ p \mid a(m) \text{ and } p \mid b(m)}} \frac{1}{m} = o\Bigg( \frac{\sqrt{\log_2 x}}{\log_4 x} \log{x}\Bigg)
\end{equation}
and
\begin{equation}\label{gcdsum}
\sum_{\substack{d\text{ squarefree} \\p \mid d \Rightarrow p \in \mathcal{P} \\ \omega(d)\le k}} \sum_{\substack{1 < m < x \\ \text{ $d$-compat} \\ \text{not $d$-ideal}}} \frac{(d,a(m))}{md} \ll (\log_2 x)^{O(1)},
\end{equation}
Theorem \ref{thm:main} is true with $f(n)$ in place of $s(n)$.
\end{prop}

\subsection{The sum of prime divisors} For each positive integer $n$, let $\beta(n):=\sum_{p\mid n} p$ denote the sum of the prime divisors of $n$. If $1 < n \leq x$, then $0 < \beta(n) \le n \leq x$. If $P$ is a prime not dividing the integer $m$, then
\[
\beta(mP) = P + \beta(m).
\]
We apply Proposition \ref{prop:ekgeneral}, with $f(n) = \beta(n)$, $a(m) = 1$, and $b(m) = \beta(m)$. Since $a(m) = 1$, one quickly observes that the sums on the left-hand sides of (\ref{smallprimessum}) and (\ref{gcdsum}) are empty. Thus, Theorem \ref{thm:main} holds with $\beta(n)$ in place of $s(n)$. The same argument applies, verbatim, with $\beta(n)$ replaced by $A(n) = \sum_{p^k \parallel n} kp$, where prime factors are summed with multiplicity. For other work on the value distribution of $\beta(n)$ and $A(n)$, see \cite{hall70, hall71, hall72, alladi77, pollack14, goldfeld17}.

\subsection{A shifted divsior function.} Let $f(n) = n + \tau(n)$, where $\tau(n)$ denotes the number of divisors of $n$. Then if $n \leq x$,  $f(n) < x^{O(1)}$ trivially. If $P$ is a prime not dividing the positive integer $m$, then
\[
f(mP) = mP + \tau(mP) = Pm + 2\tau(m),
\]
so $a(m) = m$ and $b(m) = 2\tau(m)$ in this case. For $m \leq x$, the largest exponent appearing in the prime factorization of $m$, and hence the largest prime divisor of $\tau(m)$, is $\ll \log x$. This means there is no value of $m$ that is $d$-compatible but not $d$-ideal, since every prime $p \mid d$ satisfies $p > (\log x)^2$. Equation (\ref{gcdsum}) is therefore satisfied, since the sum is empty.

Equation (\ref{smallprimessum}) is handled nearly as easily. Ignoring the condition $p \mid b(m)$ in the inner sum, the left-hand side of (\ref{smallprimessum}) is at most
\begin{equation}\label{eq:ifmdiv}
\sum_{p \leq y} \sum_{\substack{m \leq x \\ p \mid m}} \frac{1}{m} \ll \log x \sum_{p \leq y} \frac{1}{p} \ll \log x \log_3 x = o\Bigg( \frac{\sqrt{\log_2 x}}{\log_4 x} \log{x}\Bigg),
\end{equation}
as desired. Thus, by Proposition \ref{prop:ekgeneral}, Theorem \ref{thm:main} holds with $f(n) = n + \tau(n)$ in place of $s(n)$.
Similar arguments apply to  $n-\tau(n)$ and $n\pm \omega(n)$. The functions $n-\tau(n)$ and $n-\omega(n)$ appear in work of Luca \cite{luca05}; for each of these two functions, he shows that the range is missing infinitely many positive integers.

\subsection{The cototient function.} Let $f(n) = n - \phi(n)$, where (as above) $\phi(n)$ is Euler's totient function. (See \cite{erdos73, BS95, FL00, GAM05, PY13, PP16} for studies of the range of $n-\phi(n)$.) Note that $0 < f(n) < n$ for $n > 1$ and, if $P$ is a prime not dividing $m$,
\[
f(mP) = Pm - \phi(Pm) = Pm - (P-1)\phi(m) = P(m - \phi(m)) + \phi(m).
\]
We apply Proposition \ref{prop:ekgeneral} with $f(n) = n - \phi(n)$, $a(m) = m - \phi(m)$, and $b(m) = \phi(m)$. We first observe that equation (\ref{smallprimessum}) can be established as in \eqref{eq:ifmdiv}, noting that if $p \mid a(m)$ and $p \mid b(m)$, then $p \mid a(m) + b(m) = m$. To show (\ref{gcdsum}), use the argument surrounding (\ref{eq:toinsertback}), replacing $s(m)$ by $a(m) = m - \phi$ and $\sigma(m)$ by $b(m) = \phi(m)$. The argument carries through with only the slightest of modifications. By Proposition \ref{prop:ekgeneral}, Theorem \ref{thm:main} holds with $f(n) = n - \phi(n)$ in place of $s(n)$.

Several other applications of our method could be given, although some  require slight changes to the framework. For example, fix an integer $a\ne 0$ and consider the shifted totient function $f(n) = \phi(n)+a$. It is not hard to prove that the hypotheses of Proposition \ref{prop:ekgeneral} are satisfied with $a(m) = \phi(m)$ and $b(m)=-\phi(m)+a$, with one exception: If $a>0$ is in the range of $\phi$, then $b(m)$ will vanish at some $m>1$. However, it is still true that $b(m)$ is nonvanishing for all $m > m_0(a)$, and one can simply run our argument with the condition  $n/P^{+}(n) > m_0(a)$ added to the definition of $\Omega$. 


\bibliographystyle{amsalpha}
\bibliography{references}

\providecommand{\bysame}{\leavevmode\hbox to3em{\hrulefill}\thinspace}
\providecommand{\MR}{\relax\ifhmode\unskip\space\fi MR }
\providecommand{\MRhref}[2]{%
  \href{http://www.ams.org/mathscinet-getitem?mr=#1}{#2}
}
\providecommand{\href}[2]{#2}
\begin{thebibliography}{{Erd}73}

\bibitem[AE77]{alladi77}
K.~Alladi and {Erd\H{o}s, P.}, \emph{On an additive arithmetic function},
  Pacific J. Math. \textbf{71} (1977), 275--294.

\bibitem[BHS05]{BHS05}
W.D. Banks, G.~Harman, and I.E. Shparlinski, \emph{Distributional properties of
  the largest prime factor}, Michigan Math. J. \textbf{53} (2005), 665--681.

\bibitem[Bil69]{billingsley69}
P.~Billingsley, \emph{On the central limit theorem for the prime divisor
  functions}, Amer. Math. Monthly \textbf{76} (1969), 132--139.

\bibitem[Bil95]{billingsley95}
\bysame, \emph{Probability and measure}, third ed., Wiley Series in Probability
  and Mathematical Statistics, John Wiley \& Sons, Inc., New York, 1995.

\bibitem[BS95]{BS95}
J.~Browkin and A.~Schinzel, \emph{On integers not of the form {$n-\phi(n)$}},
  Colloq. Math. \textbf{68} (1995), 55--58.

\bibitem[Dav33]{davenport33}
H.~Davenport, \emph{\"{U}ber numeri abundantes}, S.-{B}er. {P}reuß. {A}kad.
  {W}iss., math.-nat. {K}l. (1933), 830–837.

\bibitem[EK40]{EK40}
P.~Erd\H{o}s and M.~Kac, \emph{The {G}aussian law of errors in the theory of
  additive number theoretic functions}, Amer. J. Math. \textbf{62} (1940),
  738--742.

\bibitem[{Erd}73]{erdos73}
P.~{Erd\H{o}s}, \emph{\"{U}ber die {Z}ahlen der {F}orm {$\sigma (n)-n$} und
  {$n-\phi(n)$}}, Elem. Math. \textbf{28} (1973), 83--86.

\bibitem[FL00]{FL00}
A.~Flammenkamp and F.~Luca, \emph{Infinite families of noncototients}, Colloq.
  Math. \textbf{86} (2000), 37--41.

\bibitem[GM05]{GAM05}
A.~Grytczuk and B.~{M{\polhk{e}}dryk}, \emph{On a result of
  {F}lammenkamp-{L}uca concerning noncototient sequence}, Tsukuba J. Math.
  \textbf{29} (2005), 533--538.

\bibitem[Gol17]{goldfeld17}
D.~Goldfeld, \emph{On an additive prime divisor function of {A}lladi and
  {{E}rd\H{o}s}}, Analytic number theory, modular forms and
  {$q$}-hypergeometric series, Springer Proc. Math. Stat., vol. 221, Springer,
  Cham, 2017, pp.~297--309.

\bibitem[Hal70]{hall70}
{R.R.} Hall, \emph{On the probability that {$n$} and {$f(n)$} are relatively
  prime}, Acta Arith. \textbf{17} (1970), 169--183, corrigendum in \textbf{19}
  (1971), 203--204.

\bibitem[Hal71]{hall71}
\bysame, \emph{On the probability that {$n$} and {$f(n)$} are relatively prime.
  {II}}, Acta Arith. \textbf{19} (1971), 175--184.

\bibitem[Hal72]{hall72}
\bysame, \emph{On the probability that {$n$} and {$f(n)$} are relatively prime.
  {III}}, Acta Arith. \textbf{20} (1972), 267--289.

\bibitem[HR00]{HR00}
G.H. Hardy and S.~Ramanujan, \emph{The normal number of prime factors of a
  number {$n$} [{Q}uart. {J}. {M}ath. {\bf 48} (1917), 76--92]}, Collected
  papers of {S}rinivasa {R}amanujan, AMS Chelsea Publ., Providence, RI, 2000,
  pp.~262--275.

\bibitem[LLPSR]{LPR21}
N.~Lebowitz-Lockard, P.~Pollack, and A.~Singha~Roy, \emph{Distribution mod $p$
  of {E}uler's totient and the sum of proper divisors}, submitted.

\bibitem[LP15]{LP15}
F.~Luca and C.~Pomerance, \emph{The range of the sum-of-proper-divisors
  function}, Acta Arith. \textbf{168} (2015), 187--199.

\bibitem[Luc05]{luca05}
F.~Luca, \emph{On numbers not of the form {$n-\omega(n)$}}, Acta Math. Hungar.
  \textbf{106} (2005), 117--135.

\bibitem[Pol14]{pollack14}
P.~Pollack, \emph{Some arithmetic properties of the sum of proper divisors and
  the sum of prime divisors}, Illinois J. Math. \textbf{58} (2014), 125--147.

\bibitem[Pom18]{pomerance18}
C.~Pomerance, \emph{The first function and its iterates}, Connections in
  discrete mathematics, Cambridge Univ. Press, Cambridge, 2018, pp.~125--138.

\bibitem[PP16]{PP16}
P.~Pollack and C.~Pomerance, \emph{Some problems of {{E}rd\H{o}s} on the
  sum-of-divisors function}, Trans. Amer. Math. Soc. Ser. B \textbf{3} (2016),
  1--26.

\bibitem[PPT18]{PPT18}
P.~Pollack, C.~Pomerance, and L.~Thompson, \emph{Divisor-sum fibers},
  Mathematika \textbf{64} (2018), 330--342.

\bibitem[PY14]{PY13}
C.~Pomerance and H.-S. Yang, \emph{Variant of a theorem of {{E}rd\H{o}s} on the
  sum-of-proper-divisors function}, Math. Comp. \textbf{83} (2014), 1903--1913.

\bibitem[Tro15]{troupe15}
L.~Troupe, \emph{On the number of prime factors of values of the
  sum-of-proper-divisors function}, J. Number Theory \textbf{150} (2015),
  120--135.

\bibitem[Tro20]{troupe20}
\bysame, \emph{Divisor sums representable as the sum of two squares}, Proc.
  Amer. Math. Soc. \textbf{148} (2020), 4189--4202.

\end{thebibliography}
\end{document}